\documentclass[11pt]{article}

\usepackage{amsmath, amssymb, amsthm, amsfonts, mathtools}
\usepackage{xcolor}
\usepackage{hyperref}
\usepackage{cite}

\definecolor{webgreen}{rgb}{0,.5,0}
\definecolor{webbrown}{rgb}{.6,0,0}

\newtheorem{theorem}{Theorem}

\newtheorem{lemma}[theorem]{Lemma}

\title{The Eighth Power Moments of $\Delta(x)$}
\author{Junhao Liao\quad \quad\quad Junjie Liao\\
}
\date{}
\begin{document}
	\maketitle
	\noindent
	\textbf{Abstract}. Using Voronoi's truncated formula for $\Delta(x)$ involving Bessel functions, the first author derives an asymptotic formula for the eighth-power moments with an error term of order $O\left(X^{3 - \frac{1}{254} + \varepsilon}\right)$.
	\noindent
	\textbf{MR(2000) Subject Classification}. 11L07, 11B83
	\thispagestyle{empty}
	\section{Introduction}
	Dirichlet's divisor problem concerns calculating the sum of the number of divisors of all positive integers $n$ less than a real number $x$ ($x \geqslant 1$). When studying the divisor problem, Dirichlet \cite{1} first used a partitioning method and obtained:
	\begin{align}
		D(x)=x\log x+(2\gamma-1)x+O(\sqrt{x}),
	\end{align}
	thereby finding an upper bound for $\Delta(x)$, where 
	\begin{align}
		\Delta(x)=D(x)-x\log x-(2\gamma-1)x,
	\end{align}
	and $\gamma$ is the Euler-Mascheroni constant. Many subsequent studies attempted to improve the upper bound of $\Delta(x)$ using new methods, but there is still a gap from the conjecture $\Delta(x)\ll x^{1/4+\varepsilon}$. If we let $\Delta(x)\ll x^{\alpha}(\log x)^{\beta}$, the existing results are:
	$$
	\begin{aligned}
		& \begin{array}{l}
			\alpha=\frac{1}{3} \sim 0.33333 \quad\quad\quad\quad\quad\quad\quad\ \  \beta=1 \quad\quad\quad\quad  \text { (Voronoi \cite{27}, 1904) } \\
			\alpha=\frac{27}{82} \sim 0.32927\quad\quad\quad\quad\quad\quad  \ \ \ \ \beta=\frac{11}{41} \quad\quad\quad\quad  \text { (Van der Corput \cite{30}, 1928) }
		\end{array} \\
		& \  \alpha=\frac{35}{108}+\varepsilon \sim 0.32407 \quad\quad\quad\quad\ \ \ \beta=0 \quad\quad\quad\quad  \ \  \text { (Kolesnik \cite{31}, 1982) } \\
		&\  \alpha=\frac{131}{416} \sim 0.31490\quad\quad\quad\quad\quad\quad\ \beta=\frac{26947}{8320}  \quad\quad \ \ \text { (Huxley \cite{2}, 2003) } \\
		&\ \alpha=\frac{517}{1648} +\varepsilon\sim0.31371 \quad\quad\quad\ \  \ \ \beta=0  \quad\quad\quad\quad\quad  \text {(Bourgain and Watt \cite{28}, 2017) }
	\end{aligned}
	$$
	To further understand the properties of $\Delta(x)$, researchers have also studied the $k$-th power integral mean of $\Delta(x)$, i.e., the asymptotic formula for $\int_1^{X} \Delta^k(x) d x$ where $k=1,2,3, \cdots$.
	
	When calculating the $k$-th power integral mean of $\Delta(x)$, the following conclusions can be obtained:
	\\First, in the early 20th century, Voronoi \cite{27} first proved:
	\begin{align}
		\int_0^{X} \Delta(x) d x=\frac{X}{4}+O(X^{3/4}),
	\end{align}
	Subsequently, Cramér \cite{29} obtained the asymptotic formula for the second power mean:
	\begin{align}
		\int_1^X \Delta^2(x) d x=\frac{(\zeta(3 / 2))^4}{6 \pi^2 \zeta(3)} X^{3 / 2}+O\left(X^{5 / 4+\varepsilon}\right).
	\end{align}
	In 1956, Dong Guangchang \cite{3} improved the asymptotic formula for the second power mean and obtained:
	\begin{align}
		\int_1^x \Delta^2(X) d x=\frac{\zeta(3 / 2)^4}{6 \pi^2 \zeta(3)} X^{3 / 2}+O\left(X \log ^5 X\right),
	\end{align}
	In 1990, Tsang \cite{6} began studying the cases $k=3$ and $k=4$, obtaining the following asymptotic formulas:
	\begin{align}\label{1}
		& \int_2^T \Delta^3(x) d x = \frac{3 C_1}{28 \pi^3} T^{7 / 4}+O\left(T^{7 / 4-\delta_3+\varepsilon}\right),\\
		\label{2}& \int_2^T \Delta^4(x) d x=\frac{3 C_2}{64 \pi^4} T^2+O\left(T^{2-\delta_4+\varepsilon}\right),
	\end{align}
	where $\delta_3=1 / 14$, $\delta_4=1 / 23$,
	\begin{equation}
		\begin{aligned}
			C_1=\sum_{\alpha, \beta ,h\in \mathbb{N}}
			&(\alpha \beta(\alpha+\beta))^{-3 / 2} h^{-9 / 4}|\mu(h)| d\left(\alpha^2 h\right) \\
			&d\left(\beta^2 h\right) d((\alpha+\beta)^2 h),
		\end{aligned}
	\end{equation}
	\begin{align}
		C_2=\sum_{\substack{n, m, l, k \in \mathbb{N}\\
				\sqrt{n}+\sqrt{m}=\sqrt{k}+\sqrt{l}}}(n m k l)^{-3 / 4} d(n) d(m) d(k) d(l).
	\end{align} \\
	Thereafter, Zhai\cite{7} used Tsang's method to prove that $\delta_3=1 / 4$ and $\delta_4=2/41$ hold. In 2004, Ivić and Sargos \cite{8} also derived the asymptotic formulas (\ref{1}) and (\ref{2}), with their $\delta_3$ and $\delta_4$ being 7/20 and 1/12, respectively, improving Tsang's results \cite{6}. Similarly, for the generalized expression of the mean integral of $\Delta(x)$, the following conclusion holds:
	\\In 1985, Ivić \cite{4} proved the estimate:
	\begin{align}\label{3}
		\int_1^X  |\Delta (x)|^A d x \ll T^{1+\frac{A}{4}+\varepsilon} \ (0 \leq A \leq 35 / 4).
	\end{align}
	In 2003, Zhai  \cite{9} also obtained the generalized expression for the mean integral of $\Delta(x)$, stating that for any fixed real number $\frac{267}{27}>A>9$ and $3 \leqslant k\leqslant A$,
	\begin{align}
		\int_2^T \Delta^k(x) d x=C_k T^{1+k / 4}+O\left(T^{1+k / 4-\delta_k+\varepsilon}\right)
	\end{align}
	holds, where $C_k$ and $0<\delta_k<1+k / 4$ are constants expressible explicitly. When seeking the $k$-th power mean estimation of $\Delta(x)$, i.e., the asymptotic formula for $\int_1^{X} \Delta^k(x) d x$ ($k=1,2,3, \cdots$), the main tool used is the truncation formula of $\Delta(x)$ proved by Voronoi \cite{27} in the last century:
	\begin{align}
		\Delta(x)=-\frac{2\sqrt{x}}{\pi}\sum_{n=1}^{\infty}\frac{d(n)}{\sqrt{n}}\left(K_1(4\pi\sqrt{nx})+\frac{\pi }{2} Y_1(4\pi\sqrt{nx})\right),
	\end{align}
	where $K_1$ and $Y_1$ are Bessel functions, and the commonly used truncation formula is:
	\begin{align}\label{103}
		\Delta(x)=\frac{x^{1/4}}{\pi\sqrt{2}}\sum_{n\leqslant N}\frac{d(n)}{n^{3/4}}\cos(4\pi\sqrt{nx}-\pi/4)+O\left(X^{1 / 2+\varepsilon}N^{-1/2}\right).
	\end{align}
	The above methods and conclusions have greatly assisted this paper in deriving the asymptotic formula for $\int_2^{X} \Delta^8(x) d x$. For the first time, this paper obtains the asymptotic formula for $\int_2^{X} \Delta^8(x) d x$ and the following conclusions.
	\begin{theorem}\label{80}
		For any fixed integer $X\geqslant 10$, we have
		$$
		\int_2^{X} \Delta^8(x) d x=\frac{35C_7-28C_4}{2048\pi^{8}}\int_2^{X} x^2 d x+O\left(X^{3-1/254+\varepsilon}\right),
		$$ where
		\begin{equation}
			\begin{aligned}
				C_4=\sum_{\substack {n,m,k,l,r,s,t,j\in \mathbb{N}\\	\sqrt{n}+\sqrt{m}+\sqrt{k}+\sqrt{l}+\sqrt{r}+\sqrt{s}=\sqrt{t}+\sqrt{j}}}
				&(n m k l rstj)^{-3 / 4} d(n) d(m) 
				\\&d(k) d(l) d(r) d(s) d(t) d(j)			
			\end{aligned}
		\end{equation}
		\begin{equation}
			\begin{aligned}
				C_7=\sum_{\substack{n, m, l, k,r,s,t,j \in \mathbb{N}\\
						\sqrt{n}+\sqrt{m}+\sqrt{k}+\sqrt{l}=\sqrt{r}+\sqrt{s}+\sqrt{t}+\sqrt{j}}}
				&(n m k l rstj)^{-3 / 4} d(n) d(m) 
				\\&d(k) d(l) d(r) d(s) d(t) d(j).
			\end{aligned} 
		\end{equation}
	\end{theorem}
	\begin{theorem}\label{81}
		For any fixed $0<\delta<\frac{25}{32}$, there exists $k>0$ such that
		\begin{equation}
			\begin{aligned}
				\int_X^{X+H} \Delta^8(x) d x=C\left((X+H)^3-X^3\right)\left(1+O\left(X^{-k}\right)\right),
			\end{aligned}
		\end{equation}
		holds, where $X^{\frac{7}{32}+\delta} \leqslant H\leqslant X$ and $C=\frac{1}{3}(\pi \sqrt{2})^{-8}\left(\frac{35C_7}{128} -\frac{7 C_4}{32}\right)$.
	\end{theorem}
	\section{Main Lemma}
	Before proving Theorem ${\ref{80}}$, we first present several common lemmas. In the paper, Lemmas \ref{4}-\ref{6} correspond to \cite{6}[Lemma 1,2,3] respectively, 
	and Lemmas \ 
	\ref{7},\ref{8} correspond to \cite{8}[Lemma 4,5]. 
	In this section, Theorem ${\ref{80}}$ is derived from Lemma 2.13.
	\begin{lemma}\label{4}
		Let $n, m, k, l$ be natural numbers, and $\sqrt{n}+\sqrt{m} \pm \sqrt{k}-\sqrt{l} \neq 0$, then
		\begin{align}
			|\sqrt{n}+\sqrt{m} \pm \sqrt{k}- \sqrt{l}| \gg \max (n, m, k, l)^{-\frac{7}{2}}.
		\end{align}
	\end{lemma}
	\begin{proof}
		See Lemma 2 in Tsang \cite{6}. Similarly, from Lemma 1 in Tsang \cite{6}, we obtain:
		\begin{align}
			\left|\sum_{i=1}^k\left(\pm \sqrt{n_i}\right)\right| \gg \left[\max _{1\leqslant i\leqslant k}\left(n_i\right)\right]^{(1-2^{k-1}) / 2}.
		\end{align}
	\end{proof}
	\begin{lemma}\label{5}
		Let $n, m, k, l, r, s, t, j$ be natural numbers, and suppose $$\sqrt{n}+\sqrt{m}+\sqrt{k}+\sqrt{l} \pm \sqrt{r}-\sqrt{s}-\sqrt{t}-\sqrt{j} \neq 0,$$ then
		\begin{align}
			|\sqrt{n}+\sqrt{m}+\sqrt{k}+\sqrt{l} \pm \sqrt{r}-\sqrt{s}-\sqrt{t}-\sqrt{j}| \gg \max (n, m, k, l, r, s, t, j)^{-127/2}.
		\end{align}
	\end{lemma}
	\begin{lemma}\label{6}
		Let $g(x)$ and $h(x)$ be continuous real functions, with $g(x)$ monotonic, then
		\begin{align}
			\int_a^b g(x) h(x) d x \ll\left(\max _{a \leqslant x \leqslant b}|g(x)|\right)\left(\max _{a \leqslant u<v \leqslant b}\left|\int_u^v h(x) d x\right|\right).
		\end{align}
	\end{lemma}
	\begin{lemma}\label{7}
		Let $N$ denote the number of positive integer solutions $(m, n, k, l)$ to the inequality
		$$
		|\sqrt{m}+\sqrt{n}-\sqrt{k}-\sqrt{l}| \leqslant \delta \sqrt{k}\ (\delta>0),
		$$
		then
		\begin{align}
			N \ll_{\varepsilon} K M M L\left(\delta+K^{\varepsilon-3 / 2}\right)+\sqrt{K M M^{\prime} L}
		\end{align}
		or
		$$
		N \ll_{\varepsilon} K M M L\left(\delta K^2+\left(K M M^{\prime} L\right)^{-1 / 2}\right) K^{\varepsilon}.
		$$
	\end{lemma}
	\begin{lemma}\label{8}
		For real numbers $0<\delta<\frac{1}{2}$, $\beta$, and $\alpha \gg 1$, we have
		\begin{align}
			\#\{K<k \leqslant 2 K \mid \|\beta+\alpha \sqrt{k}\|<\delta\} \ll k \delta+|\alpha|^{1 / 2} k^{1 / 4+\varepsilon}+k^{1/2+\varepsilon}.
		\end{align}
	\end{lemma}
	\begin{lemma}\label{9}
		Let $1 \leqslant N \leqslant M \leqslant K \leqslant L$, $1 \leqslant R \leqslant S \leqslant T \leqslant J$, $L \asymp J$, and suppose $\sqrt{n}+\sqrt{m}+\sqrt{k}+\sqrt{l}\pm\sqrt{r} \neq \sqrt{s}+\sqrt{t}+\sqrt{j}$, with $m \sim M$, $n \sim N$, $k \sim K$, $l \sim L$, $r \sim R$, $S \sim S$, $t \sim T$, $j\sim J$. Let $A_{\pm}(N, M, K, L, R, S, L, J)$ denote the number of solutions to
		$$0<|\Delta_{\pm}=\sqrt{n}+\sqrt{m}+\sqrt{k}+\sqrt{l} \pm \sqrt{r}-\sqrt{s}-\sqrt{t}-\sqrt{j}|<\delta \sqrt{L}\ (\delta>0),$$
		then
		$$A_{\pm}(N, M, K, L, R, S, T, J) \ll\left(\delta L^2+L^{\frac{1}{2}+\varepsilon}\right)NMKRST,$$ 
		In particular, when $\delta\gg \frac{1}{L}$, we have $A_{\pm}(N, M, K, L, R, S, T, J) \ll \delta NMKLRSTJ$.
	\end{lemma}
	\vspace{-1.5em} 
	\begin{proof}
		Let $u_{\pm}=\sqrt{n}+\sqrt{m}+\sqrt{k} \pm \sqrt{r}-\sqrt{s}-\sqrt{t}$, then we can get
		$$\Delta_{\pm}^2+2 \sqrt{j} \Delta_{\pm}=l-j+2 \sqrt{l} u_{\pm}+\left(u_{\pm}\right)^2+O(\delta L).$$ 
		From $\left|\Delta_{\pm}\right|<\delta \sqrt{L}$, we have
		\begin{align}\label{10}
			\left|l-j+2 \sqrt{l} u_{\pm}+u_{\pm}\right|\leqslant\left(\delta^2+2 \delta\right) L \ll \delta L
		\end{align}
		Thus $A_{\pm}(N, M, K, L, R, S, T, J)$ does not exceed the number of solutions to inequality (\ref{10}), which can be found by case analysis. \\
		If $\delta \gg\frac{1}{L}$, then
		$$A_{\pm}(N, M, K, L, R, S, T, J) \ll \delta N M K L R S T J.$$
		If $\delta \leqslant \frac{1}{2 L}$, we have
		\begin{align}\label{11}
			0<\left\|j-l-2 u_{\pm} \sqrt{l}-u_{\pm}^2\right\|=\left\|u_{\pm}^2+2 u_{\pm} \sqrt{l}\right\|<\delta L \leqslant \frac{1}{2}.
		\end{align}
		Thus, for a fixed array $(n,m,k,l,r,s,t,j)$, there is at most one $l$ satisfying (\ref{11}). If such an $l$ exists, let $\alpha=2 u_{\pm}$, $\beta=u_{\pm}^2$, and the number of solutions to (\ref{11}) is $B_{\pm}(N, M, K, R, S, T)$.  
		\\When $1\ll\alpha\ll \sqrt{L}$, by Lemma $\ref{4}$,
		$$
		B_{\pm}(N, M, K, R, S, T) \ll\left(\delta L^2+L^{\frac{1}{2}+\varepsilon}\right) N M K R S T,
		$$
		When $\alpha \leqslant \frac{2}{\sqrt{L}}$, since $J \asymp L$, if $u_{\pm}=0$, then $j=l$ and $\Delta_{\pm}=0$, contradicting $\left|\Delta_{\pm}\right|>0$. Similarly, using the above method, we get:
		$$
		\begin{aligned}
			& \left|2 \sqrt{k}(\sqrt{n}+\sqrt{m} \pm \sqrt{r}-\sqrt{s})+(\sqrt{n}+\sqrt{m} \pm \sqrt{r}-\sqrt{s})^2+k-t\right|<\delta_1 \sqrt{\frac{K}{L}},
		\end{aligned}
		$$
		where $\delta_1$ is a constant, so 
		$$B_{\pm}(N, M, K, R, S, T) \ll \log _2 T\left(1+\delta_1 K/ L\right) K^{\frac{1}{2}+\varepsilon} NMSR.$$
		We then classify $\delta_1$: When $\delta_1 \leqslant \frac{1}{2} \sqrt{\frac{L}{K}}$,
		\begin{align}\label{12}
			0<\left\|2 \sqrt{k}(\sqrt{n}+\sqrt{m} \pm \sqrt{r}-\sqrt{s})+(\sqrt{n}+\sqrt{m} \pm \sqrt{r}-\sqrt{s})^2\right\|<\delta_1 \sqrt{\frac{K}{L}}.
		\end{align}
		Thus, except for fixed $(n, m, r, s)$, there is at most one $k$ satisfying (\ref{12}). If such a $k$ exists, we conclude: \\
		When $|\sqrt{n}+\sqrt{m} \pm \sqrt{r}-\sqrt{s}| \gg 1$, by Lemma 1.5,
		$$	\begin{aligned}
			B_{\pm}(N, M, K, R, S, T)&  \ll\log_2 T\left(\delta_1 \sqrt{\frac{K}{L}} K+K^{\frac{1}{2}+\varepsilon}\right) N M R S \\& \ll N M R S\left(K+K^{\frac{1}{2}+\varepsilon}\right),
		\end{aligned}
		$$
		When $|\sqrt{n}+\sqrt{m} \pm \sqrt{r}-\sqrt{s}|<K^{-1/2}$, where $\delta_2$ satisfies $R^{-7/2}<\delta_2 \ll K^{-1/2}$ by Lemma $\ref{4}$, then by Lemma $\ref{7}$ and the dichotomy method,
		$$
		\begin{aligned}
			B_{\pm}(N, M, K, R, S, T)&\ll\log _2 T \log _2 R \sum_{\delta_2 \leqslant |\sqrt{n}+\sqrt{m}-\sqrt{r}-\sqrt{s}| \leqslant 2 \delta_2}1\\
			& \ll \log _2 T\log _2 R\left[\left(\delta_2 R^{-\frac{1}{2}}+R^{\varepsilon-\frac{3}{2}}\right) R M N S+\sqrt{R M N S}\right] \\
			& \ll T^{\varepsilon}(R M N S)^{1+\varepsilon}.
		\end{aligned}
		$$
		In summary, 
		$$A_{\pm}(N, M, K, L, R, S, T, J) \ll\left(\delta L^2+L^{\frac{1}{2}+\varepsilon}\right) M N K R S T,$$
		and particularly when $\delta\gg\frac{1}{L}$,
		$$A_{\pm}(N, M, K, L, R, S, T, J) \ll \delta N M K L R S T J.$$
		Hence, the lemma follows.
	\end{proof}
	\begin{lemma}\label{15}
		Let $1 \leqslant N \leqslant M \leqslant K \leqslant L$, $1 \leqslant R \leqslant S \leqslant T \leqslant J$, $L \asymp J$, and suppose $\sqrt{n}+\sqrt{m}+\sqrt{k}+\sqrt{l}+ \sqrt{r}+\sqrt{s}\pm\sqrt{t}\neq\sqrt{j}$, with $m \sim M$, $n \sim N$, $k \sim K$, $l \sim L$, $r \sim R$, $S \sim S$, $t \sim T$, $j\sim J$. Let $A_{\pm}(N, M, K, L, R, S, L, J)$ denote the number of solutions to
		$$0<|\sqrt{n}+\sqrt{m}+\sqrt{k}+\sqrt{l} + \sqrt{r}+\sqrt{s}\pm\sqrt{t}-\sqrt{j}|<\delta \sqrt{L}\ (\delta>0),$$
		then
		$$A_{\pm}^{'}(N, M, K, L, R, S, T, J) \ll\left(\delta L^2+L^{\frac{1}{2}+\varepsilon}\right)NMKRST.$$ 
		When $\delta\gg \frac{1}{L}$, we have $A_{\pm}^{'}(N, M, K, L, R, S, T, J) \ll \delta NMKLRSTJ$.
	\end{lemma}
	\begin{proof}
		The proof is identical to that of Lemma $\ref{9}$.
	\end{proof}
	\begin{lemma}\label{13}
		Let $q$ be a positive integer, and let $Q=2^q$. If $H \leqslant|I|$, $H_q=H$, $H_{q-1}=H^{\frac{1}{2}}$, $\cdots$, $H_1=H^{\frac{2}{Q}}$, then
		$$|S|^Q \leq 8^{Q-1}\left\{\frac{|I|^Q}{H}+\frac{|I|^{Q-1}}{H_1 \cdots H_q} \sum_{1 \leqslant h_1 \leqslant H_1} \cdots \sum_{1 \leqslant h_q \leqslant H_q}\left|S_q(\mathbf{h})\right|\right\}$$
		where $\mathbf{h}=(h_1, h_2, \cdots, h_q)$, $S_q(\mathbf{h})=\sum\limits_{n\in I(\mathbf{h})} e\left(f_q(n ;\mathbf{h} )\right)$, $S=\sum\limits_{n \sim N} e(n)$,
		$$
		\begin{aligned}
			& f_q(n ; \mathbf{h})=\int_0^1 \cdots \int_0^1 \frac{\partial^2}{\partial t_1 \partial t_2 \cdots \partial t_q} f(n+\mathbf{h} \cdot \mathbf{t}) d t_1 \cdots d t_q, \\
			& \mathbf{t}=(t_1, \cdots, t_q), \text{ and } I(h)=(a, b-h_1-h_2-\cdots-h_q].
		\end{aligned}
		$$
	\end{lemma}
	\begin{proof}
		See Lemma 2.8 in Kolesnik and Graham \cite{1}.
	\end{proof}
	\begin{lemma}\label{14}
		Let $N>2$, and $S(x, N, k)=\sum\limits_{n \sim N} e\left(x n^{1/k}\right)$ , then we can show that
		$$
		\begin{aligned}
			&\int_U^{2 U}|S(x, N, k)|^8 d x \ll\left(U N^4+N^{8-1/k}\right) N^{\varepsilon}.
		\end{aligned}	
		$$
	\end{lemma}
	\begin{proof}
		From Lemma $\ref{13}$, take $H_1=N$, $I=N$, $q=3$, and let $H_3=H$, $H_2=H^{\frac{1}{2}}$, $H_1=H^{\frac{1}{4}}$, where
		$$
		\begin{aligned}
			&S_3(x, N, k, \mathbf{h})=\sum_{n \in I(\mathbf{h})} e\left(f_3(n ;\mathbf{h} )\right),\\
			&I(\mathbf{h})=(N, 2N-h_1-h_2-h_3],\\
			&f_3(n ;\mathbf{h} )=\int_0^1 \int_0^1 \int_0^1 \frac{\partial^3f(n+\mathbf{h} \cdot \mathbf{t})}{\partial t_1 \partial t_2 \partial t_2} d t_1 d t_2 d t_3.
		\end{aligned}
		$$
		Then
		$$
		\begin{aligned}
			|S(x, N, k)|^8 \leqslant& 8^7\left\{N^4+\sum_{1 \leqslant h_1 \leqslant H} \sum_{1 \leqslant h_2 \leqslant H^{1/2}}\sum_{1 \leqslant h_3 \leqslant H^{1/4}}\left|S_3(x, N, k, \mathbf{h})\right|\right\}\\
			\ll& N^4+\sum_{1 \leqslant h_1 \leqslant N} \sum_{1 \leqslant h_2 \leqslant N^2} \sum_{1 \leqslant h_3 \leqslant N^4}\left|S_3(x, N, k, \mathbf{h})\right|.
		\end{aligned}
		$$
		By the definition of $S_3(x, N, k, \mathbf{h})$, let
		$$
		\begin{aligned}
			\Delta=&\Delta\left(n, k, h_1, h_2, h_3\right)\\
			=& n^{1/k}+\left(n+h_1+h_2\right)^{1/k}+\left(n+h_1+h_3\right)^{1/k}+\left(n+h_2+h_3\right)^{1/k}\\
			&-\left(n+h_1\right)^k-\left(n+h_2\right)^{1/k}-\left(n+h_3\right)^{1/k}-\left(n+h_1+h_2+h_3\right)^{1/k},
		\end{aligned}
		$$
		thus
		$$
		\begin{aligned}
			&\int_{U }^{2 U} \sum_{1\leqslant h_1 \leqslant N} \sum_{1 \leqslant h_2 \leqslant  N^2} \sum_{1\leqslant h_3 \leqslant N^4} e(x \Delta) d x\\
			& = \sum_{1\leqslant h_1 \leqslant N} \sum_{1 \leqslant h_2 \leq N^2} \sum_{1\leqslant h_3 \leqslant N^4} \int_{U }^{2 U}e(x \Delta) d x \\
			& \ll \sum_{1 \leqslant h_1 \leqslant N} \sum_{1 \leqslant h_2 \leqslant N^2} \sum_{1 \leqslant h_3 \leqslant N^4}\frac{1}{|\Delta|}  \\
			&\ll N^{8-1/k+\varepsilon}.
		\end{aligned}
		$$
		The last step uses $|\Delta| \neq 0$, $\frac{1}{|\Delta|} \ll N^{-1/k+\varepsilon}$, and the number of solutions with $\Delta\neq0$ being $O(N^{8})$, hence	$\int_U^{2 U}|S(x,N, k)|^8 d x \ll\left(U N^4+N^{8-1/k}\right) N^{\varepsilon}.$ 
	\end{proof}
	\begin{lemma}\label{16}
		Let $b, \delta$ be real numbers with $0<\delta<b/ 4$, and let $l$ be an integer. There exists an $l$-times continuously differentiable function $\varphi(x)$ such that
		$$
		\begin{cases}
			\varphi(x)=1, & |x|<b-\delta \\
			0<\varphi(x)<1, & b-\delta<|x|<b+\delta \\
			\varphi(x)=0, & |x| > \delta+b
		\end{cases}
		$$ 
		The Fourier expansion of $\varphi(x)$ is $\phi(x)=\int_{-\infty}^{+\infty} e(-x y) \varphi(y) d y$, satisfying
		$$|\varphi(x)| \leqslant \min\left(2b, \frac{1}{\pi|x|}, \frac{1}{\pi|x|}\left(\frac{l}{2 \pi|x| \delta}\right)^l\right).$$
	\end{lemma}
	\begin{proof}
		See Lemma 7 in Tolev \cite{14}.
	\end{proof}
	\nointerlineskip 
	\begin{lemma}\label{17}
		\vspace{-2em} 
		Let $N_j\geqslant2$ ($j=1,2,\dots,8$), $\delta>0$, $n_j\sim N_j$. The number of solutions to
		$$0<\left|n_1^{1/k}+n_2^{1/k}+n_3^{1/k}+n_4^{1/k} \pm n_5^{1/k}-n_6^{1/k}-n_7^{1/k}-n_8^{1/k}\right|<\delta$$
		is
		$$
		A_{\pm}^{''}\left(N_1, N_2, N_3, N_4, N_5, N_6, N_7, N_8\right) \ll \prod_{j=1}^8\left(\delta^{\frac{1}{8}} N_j^{\frac{8-1/k}{8}}+N_j^{\frac{1}{4}}\right) N_j^{\varepsilon}.
		$$
	\end{lemma}
	\nointerlineskip 
	\begin{proof}
		Substituting into Lemma $\ref{16}$, we obtain the expression for $\varphi(y)$:
		$$
		\begin{aligned}
			& \left\{\begin{array}{l}
				\varphi(y)=1, \quad|y| \leq \sigma \\
				0 < \varphi(y)<1, \quad \sigma<|y|<\frac{12 \sigma}{7} \\
				\varphi(y)=0, \quad|y|\geq\frac{12\sigma}{7}
			\end{array}\right. \\
		\end{aligned}
		$$	
		where $\sigma=7\delta$, $b=\frac{5\sigma}{7}$, $l=\lfloor\log N_1 N_2 N_3 N_4 N_5 N_6 N_7 N_8\rfloor$. \\
		Let $\phi(x)=\int_{-\infty}^{+\infty} e(-x y) \varphi(y) d y$, then
		$$|\varphi(x)| \leqslant \min\left(\frac{10\sigma}{7}, \frac{1}{\pi|x|}, \frac{1}{\pi|x|}\left(\frac{7l}{2 \pi|x| \delta}\right)^l\right).$$
		Define
		$$
		R_{\pm}=\sum_{\substack{n_j\sim N_j \\
				j=1,2,3,4,5,6,7,8}} \varphi\left(n_1^{1 / k}+n_2^{1 / k}+n_3^{1 / k}+n_4^{1 / k} \pm n_5^{1 / k}-n_6^{1 / k}-n_7^{1 / k}-n_8^{1 / k}\right),
		$$
		Obviously, $A_{\pm}^{''}\left(N_1, N_2, N_3, N_4, N_5, N_6, N_7, N_8\right)$ (abbreviated as $A_{\pm}^{''}$) satisfies $A_{\pm}^{''} \leqslant R_{\pm}$. \\
		First, consider $R_{-}$ ( $R_{+}$ is similar). Let $S(x, N_j, k)=S_j$. By Hölder's inequality:
		$$
		\begin{aligned}
			R_{-}&=\sum_{\substack{n_j \sim N_j \\
					j=1,2,3,4,5,6,7,8}} \int_{-\infty}^{+\infty} \phi(x) e\left(n_1^{1 / k}+\cdots+n_4^{1 / k}-n_5^{1 / k}-n_6^{1 / k}-n_7^{1 / k}-n_8^{1 / k}\right) d x \\
			& =\int_{-\infty}^{+\infty} \phi(x) S(x, N_1, k)\cdots  S(x, N_4, k) \overline{S(x, N_5, k)\cdots S(x, N_8, k)} d x \\
			& =\int_{-\infty}^{+\infty} \phi(x) S_1 S_2 S_3 S_4\overline{S_5 S_6 S_7 S_8} d x \\
			& \leqslant \prod_{j=1}^8\left(\int_{-\infty}^{+\infty}\left|S(x, N_j, k)\right|^8|\phi(x)| d x\right)^{\frac{1}{8}}.
		\end{aligned}
		$$
		It suffices to compute $Y(x,N_j,k)=\int_0^{+\infty}\left|S(x,N_j,k)\right|^{8}\left|\phi(x)\right| dx$. For any $j$,
		$$\int_p^{+\infty}Y(x,N_j,k) d x \ll \int_{p}^{+\infty} \frac{1}{|x|^{l+1}} d x \ll 1,$$
		where $p:=\frac{1000\pi l}{\delta}$. This follows from $\left|\phi(x)\right|\leqslant\frac{1}{\pi|x|}\left(\frac{7l}{2 \pi|x| \delta}\right)^l$. \\
		Assume $p_0<p$, so $\delta U \ll \delta p\ll l$. By Lemma $\ref{14}$,
		$$
		\begin{aligned}
			\int_0^{p_0}Y(x,N_j,k)d x 
			&\ll \delta U\log p_0 \max _{0 < U \leqslant p_0} \int_U^{2 U}\left|S(x, N_j, k)\right|^8 d x  \\
			&\ll \delta\left(U N_{j}^4+N_{j}^{8-1/k}\right) N_{j}^{\varepsilon}\\
			& \ll\left(N_{j}^4+\delta N_{j}^{8-1/k}\right)N_{j}^{\varepsilon}l.
		\end{aligned}
		$$
		Similarly,
		$$
		\begin{aligned}
			\int_{p_0}^pY(x,N_j,k)d x 
			&\ll \delta U \log p \max _{p_0 \leqslant U \leqslant p} \int_U^{2 U}\left|S(x, N_j, k)\right|^8 d x  \\
			& \ll\left(N_{j}^4+\delta N_{j}^{8-1/k}\right)N_{j}^{\varepsilon}l.
		\end{aligned}
		$$
		Hence,
		$$
		\begin{aligned}
			\int_0^{+\infty}Y(x,N_j,k) dx&=\int_p^{+\infty}Y(x,N_j,k)d x
			+\int_{p_0}^{p}Y(x,N_j,k) d x
			+\int_0^{p_0}Y(x,N_j,k) d x\\
			& \ll\left(N_{j}^4+\delta N_{j}^{8-1/k}\right)N_{j}^{\varepsilon}l.\\
		\end{aligned}
		$$
		In summary, $R_{-} \ll \prod\limits_{j=1}^8\left(N_{j}^{1/2}+\delta^{1/8} N_{j}^{\frac{8-1/k}{8}} \right)N_{j}^{\varepsilon}$. Similarly,
		$$
		\begin{aligned}
			R_{+}&=\sum_{\substack{n_j \sim N_j \\
					j=1,2,3,4,5,6,7,8}} \int_{-\infty}^{+\infty} \phi(x) e\left(n_1^{1 / k}+\cdots+n_4^{1 / k}+n_5^{1 / k}-n_6^{1 / k}-n_7^{1 / k}-n_8^{1 / k}\right) d x \\
			& =\int_{-\infty}^{+\infty} \phi(x) S_1 S_2 S_3 S_4S_5\overline{ S_6 S_7 S_8} d x \\
			& \leqslant \prod_{j=1}^8\left(\int_{-\infty}^{+\infty}\left|S(x, N_j, k)\right|^8|\phi(x)| d x\right)^{\frac{1}{8}}.
		\end{aligned}
		$$
		By similar arguments, $R_{+} \ll \prod\limits_{j=1}^8\left(N_{j}^{1/2}+\delta^{1/8} N_{j}^{(8-1/k)/8} \right)N_{j}^{\varepsilon}$, completing the proof.
	\end{proof}
	\begin{lemma}\label{50}
		Let $N_j\geqslant2$ ($j=1,2,\dots,8$), $\delta>0$, $n_j\sim N_j$. The number of solutions to
		$$0<\left|n_1^{1/k}+n_2^{1/k}+n_3^{1/k}+n_4^{1/k} + n_5^{1/k}+n_6^{1/k} \pm n_7^{1/k}-n_8^{1/k}\right|<\delta$$
		is
		$$
		A_{\pm}^{'''}\left(N_1, N_2, N_3, N_4, N_5, N_6, N_7, N_8\right) \ll \prod_{j=1}^8\left(\delta^{\frac{1}{8}} N_j^{\frac{8-1/k}{8}}+N_j^{\frac{1}{4}}\right) N_j^{\varepsilon}.
		$$
		\begin{proof}
			The proof is similar to that of Lemma $\ref{17}$.
		\end{proof}
	\end{lemma}
	\begin{lemma}\label{19}
		Let $A_{0} >2$ be fixed. Then
		$$
		\int_1^T|\Delta(x)|^{A_0} d x \ll T^{1+A_0 / 4+\varepsilon}.
		$$
	\end{lemma}
	\begin{proof}
		See the upper bound estimate in Tong \cite{3}.
	\end{proof}
	\begin{lemma}\label{72}
		For a fixed $H \geq 10$,
		$$\int_H^{2 H}\left|{\sum}_Y(x)\right|^8 d x=\left(\frac{35C_7}{128} -\frac{7 C_4}{32}\right) \int_H^{2 H} x^2 d x+O\left(H^{3-\frac{1}{254}+\varepsilon}\right),
		$$
		where ${\sum}_Y(x)=x^{1/4} \sum_{n \leqslant Y} d(n) n^{-3/4} \cos \left(4 \pi \sqrt{n x}-\frac{\pi}{4}\right)$ and $Y\leqslant H^{11/36}$.
	\end{lemma}
	\begin{proof}
		Define
		$$
		\begin{aligned}
			q&=q(m, n, k, l, s, r, t, j)\\
			&=d(m) d(n) d(k) d(l) d(s) d(r) d(t) d(j)(m n k l s r t j)^{-3 / 4},
		\end{aligned}
		$$
		where $1 \leq m, n, k, l, s, r, t, j \leq Y$. \\
		Then
		$$
		\begin{aligned}
			{\sum}_Y^8(x)&=x^{2}\left(\sum_{n \leq Y} d(n) n^{-\frac{3}{4}} \cos \left(4 \pi \sqrt{n x}-\frac{1}{4} \pi\right)\right)^8\\
			&=S_1(x)+S_2(x)+S_3(x)+S_4(x)+S_5(x)+S_6(x)+S_7(x),
		\end{aligned}
		$$
		where $S_1(x),S_2(x),S_3(x),S_4(x),S_5(x),S_6(x),S_7(x)$ are
		$$
		\begin{aligned}
			& S_1(x)=\frac{1}{128} \sum q x^2 \cos (4 \pi\sqrt{x}U), \\
			& S_2(x)=-\frac{1}{16} \sum_{\sqrt{n}+\sqrt{m}+\sqrt{k}+\sqrt{l}+\sqrt{r}+\sqrt{s}+\sqrt{t}\neq\sqrt{j}} q x^2 \sin (4 \pi\sqrt{x}\Delta_{+}^{\prime}), \\
			& S_3(x)=-\frac{7}{32} \sum_{\sqrt{n}+\sqrt{m}+\sqrt{k}+\sqrt{l}+\sqrt{r}+\sqrt{s}\neq\sqrt{t}+\sqrt{j}} q x^2 \cos (4 \pi\sqrt{x}\Delta_{-}^{\prime}),\\
			&S_4(x)=-\frac{7}{32} \sum_{\sqrt{n}+\sqrt{m}+\sqrt{k}+\sqrt{l}+\sqrt{r}+\sqrt{s}=\sqrt{t}+\sqrt{j}} q x^2, \\
			&S_5(x)=\frac{7}{16} \sum_{\sqrt{n}+\sqrt{m}+\sqrt{k}+\sqrt{l}+\sqrt{r}\neq \sqrt{s}+\sqrt{t}+\sqrt{j}} q x^2 \sin (4 \pi\sqrt{x}\Delta_{+}),\\
			&S_6(x)=\frac{35}{128} \sum_{\sqrt{n}+\sqrt{m}+\sqrt{k}+\sqrt{l}\neq \sqrt{r}+\sqrt{s}+\sqrt{t}+\sqrt{j}} q x^2 \cos (4 \pi\sqrt{x}\Delta_{-}),\\
			&S_7(x)=\frac{35}{128} \sum_{\sqrt{n}+\sqrt{m}+\sqrt{k}+\sqrt{l}=\sqrt{r}+\sqrt{s}+\sqrt{t}+\sqrt{j}} q x^2,
		\end{aligned}
		$$
		with
		$$U=\sqrt{m}+\sqrt{n}+\sqrt{k}+\sqrt{l}+\sqrt{r}+\sqrt{s}+\sqrt{t}+\sqrt{j},$$
		$$\Delta_{\pm}=\Delta_{\pm}(m, n, k, l, r, s, l, j)=\sqrt{m}+\sqrt{n}+\sqrt{k}+\sqrt{l}\pm\sqrt{r}-\sqrt{s}-\sqrt{t}-\sqrt{j},$$ 
		$$\Delta_{\pm}^{\prime}=\Delta_{\pm}^{\prime}(m, n, k, l, r, s, l, j)=\sqrt{m}+\sqrt{n}+\sqrt{k}+\sqrt{s}+\sqrt{l}+\sqrt{r} \pm \sqrt{t}-\sqrt{j}.$$
		If $\Delta_{-}=0$, the solutions $(n, m, k, l, r, s, t, j)$ fall into the following cases: \\
		\textbf{Case 1}: $r=n$, $m=m_*^2 h$, $k=k_*^2 h$, $l=l_*^2 h$, $s=s_{*}^2 h$, $t=t_*^2 h$, $j=j_*^2 h$, with $$m_*+k_*+l_*=s_*+t_*+j_*,\ \ \ \ \ \ \mu(h) \neq 0.$$
		There are also 16 other similar cases existing, corresponding to permuting variables in $$\sqrt{n}+\sqrt{m}+\sqrt{k}+\sqrt{l}=\sqrt{r}+\sqrt{s}+\sqrt{t}+\sqrt{j}.$$ 
		\textbf{Case 2}: $r=n$, $m=s$, $k=k_*^2 h$, $l=l_*^2 h$, $t=t_*^2 h$, $j=j_*^2 h$, with $$k_*+l_*=t_*+j_*,\ \ \ \ \  \mu(h) \neq 0.$$. \\
		There are also 35 other similar cases existing. \\
		\textbf{Case 3}: $n=n_*^2 h$, $m=m_*^2 h$, $k=k_*^2 h$, $l=l_*^2 h$, $r=r_*^2 h$, $s=s_{*}^2 h$, $t=t_*^2 h$, $j=j_*^2 h$, with $$n_*+m_*+k_*+l_*=r_*+s_*+t_*+j_*,$$ and no two variables on either side equal. 
		
		First, we can calculate the integral of $S_7(x)$ and its error term. Assume $n \sim N$, $m \sim M$, $k \sim K$, $l \sim L$, $s \sim S$, $r \sim R$, $t\sim T$, $j \sim J$, with $1 \leq N \leq M \leq K \leq L$, $1 \leq R \leq S \leq T \leq J$, $L\ll Y$, $L\asymp J$, $Y\ll H^{11/36}$. Define
		$$S=\sum_{\sqrt{n}+\sqrt{m}+\sqrt{k}+\sqrt{l}=\sqrt{r}+\sqrt{s}+\sqrt{t}+\sqrt{j}}q -\sum_{\substack{n, m, l, k,r,s,t,j \leqslant Y\\
				\sqrt{n}+\sqrt{m}+\sqrt{k}+\sqrt{l}=\sqrt{r}+\sqrt{s}+\sqrt{t}+\sqrt{j}}}q,$$
		then
		$$
		\begin{aligned}
			S&\ll\sum_{\substack{ m, l, k,s,t,j \text{ all }1\\ \sqrt{n}-\sqrt{r}=0,\ n>Y}} d^2(n) n^{-\frac{3}{2}}+\sum_{\substack{ m, l, k,s,t,j \text{ not all }1}}q\\
			& \ll_{\varepsilon} Y^{-\frac{1}{2}+\varepsilon}+\sum_{\substack{n^2 h>Y\\r^2 h>Y}} h^{-6}d^8(h)q(m^2, n^2, k ^2,l^2, s^2, r^2, t^2, j^2) \\
			& \ll _{\varepsilon} Y^{-\frac{1}{2}+\varepsilon},
		\end{aligned}
		$$
		where
		$$
		\begin{aligned}
			q(m^2, n^2, k ^2,l^2, s^2, r^2, t^2, j^2)=&(m n k l s r t j)^{-\frac{3}{2}} d\left(m^2\right) d\left(n^2\right) d\left(k^2\right)\\ &d\left(l^2\right) d\left(s^2\right) d\left(r^2\right) d\left(t^2\right) d\left(j^2\right).
		\end{aligned}
		$$
		Hence, the integral of $S_7(x)$ is
		$$
		\begin{aligned}
			\int_{H}^{2H}S_7(x)dx&=\frac{35}{128} \sum_{\sqrt{n}+\sqrt{m}+\sqrt{k}+\sqrt{l}=\sqrt{r}+\sqrt{s}+\sqrt{t}+\sqrt{j}} \int_{H}^{2H}q x^2dx\\
			&=\frac{35C_7}{128}\int_{H}^{2H} x^2dx+O(H^{205/72+\varepsilon}).
		\end{aligned}
		$$
		Similarly, the integral of $S_4(x)$ is
		$$
		\begin{aligned}
			\int_{H}^{2H}S_4(x)dx&=-\frac{7}{32} \sum_{\sqrt{n}+\sqrt{m}+\sqrt{k}+\sqrt{l}+\sqrt{r}+\sqrt{s}=\sqrt{t}+\sqrt{j}} \int_{H}^{2H}q x^2dx\\
			&=-\frac{7C_4}{32}\int_{H}^{2H} x^2dx+O(H^{205/72+\varepsilon}).
		\end{aligned}
		$$
		For the integrals of $S_5(x)$ and $S_6(x)$, by Lemma $\ref{6}$,
		$$\left|\Delta_{\pm}\right| \gg \max \left(\sqrt{n}, \sqrt{m},\sqrt{k},\sqrt{l},\sqrt{r},\sqrt{s},\sqrt{t},\sqrt{j} \right)^{-127/2}.$$
		If $\left|\Delta_{\pm}\right| \asymp \delta \sqrt{L}$ with $\delta\gg 1$, the number of solutions is $A_{\pm} \ll \delta N M K L R S TJ$, and the order of partial integrals of $S_6(x)$ and $S_5(x)$ is roughly calculated by first derivative estimation:
		$$
		\begin{aligned}
			\int_{H}^{2H}S_6(x)dx&\ll_{\varepsilon} H^{5/2+\varepsilon}\mid \Delta_{ -}\mid^{-1} \sum_{\sqrt{n}+\sqrt{m}+\sqrt{k}+\sqrt{l}\neq\sqrt{r}+\sqrt{s}+\sqrt{t}+\sqrt{j}} q\\  &\ll_{\varepsilon} H^{\frac{5}{2}+\varepsilon} Y^{\frac{3}{2}}\\
			&\ll_{\varepsilon} H^{\frac{71}{24}+\varepsilon},
		\end{aligned}
		$$	
		Similarly, $\int_H^{2 H} S_5(x) d x \ll_{\varepsilon} H^{\frac{5}{2}+\varepsilon} Y^{\frac{3}{2}}\ll_{\varepsilon} H^{\frac{71}{24}+\varepsilon}$. \\
		If ${1/L}\leqslant \delta \leqslant 1$, then $\left|\Delta_{\pm}\right| \gg L^{-\frac{1}{2}}$, so the number of solutions with $\Delta_{\pm} \neq 0$ is calculated by
		$$
		\Delta_{\pm}^2+2 \sqrt{j} \Delta_{\pm}=l-j+2 \sqrt{l} u_{\pm}+u_{\pm}^2+O(\delta L),
		$$
		where $u_{\pm}=\sqrt{n}+\sqrt{m}+\sqrt{k} \pm \sqrt{r}-\sqrt{s}-\sqrt{t}$. For fixed $(n,m,k,l,s,r,t,j)$, there are at most $\ll_{\varepsilon}\delta l$ solutions for $l$, so $\Delta_{\pm}$ has at most $\ll\delta NMKLRSTJ$ solutions, hence
		$$
		\begin{aligned}
			\int_H^{2 H} S_6(x) d x 	&\ll_\varepsilon H^{\frac{5}{2}+\varepsilon}  \mid\Delta_{-}\mid^{-1}\delta (NMKLRSTJ)^{\frac{1}{4}}\\
			&\ll H^{\frac{5}{2}+\varepsilon} Y^{\frac{3}{2}}
			\\ &\ll_{\varepsilon} H^{\frac{71}{24}+\varepsilon},\\
			\int_H^{2 H} S_5(x) d x 	&\ll_\varepsilon H^{\frac{5}{2}+\varepsilon}  \mid\Delta_{+}\mid^{-1}\delta (NMKLRSTJ)^{\frac{1}{4}}\\
			&\ll H^{\frac{5}{2}+\varepsilon} Y^{\frac{3}{2}}\\
			&\ll_{\varepsilon} H^{\frac{71}{24}+\varepsilon}.
		\end{aligned}
		$$
		When $0<\delta \ll 1 / L$, the upper bounds for integrals of $S_5(x)$ and $S_6(x)$ reduce to
		$$
		\int_{H}^{2H} {\sum_{m,n,k,l,r,s,t,j}}^{*}qx^2 e\left(\Delta_{\pm} \sqrt{x}\right) d x,
		$$
		where $\sum^{*}$ denotes $\Delta_{\pm}\neq0$. Consider $S_6(x)$ ( $S_5(x)$ is similar), and discuss two cases: \\
		$i)\ \delta \asymp|\Delta_{\pm}| L^{-1/2}$, $|\Delta_{\pm}| < H^{-1/2}$. By Lemma $\ref{9}$ and first derivative estimation,
		$$
		\begin{aligned}
			\int_H^{2 H} {\sum}^{*}qx^2 e\left(\Delta_{\pm} \sqrt{x}\right) d x &\ll_\varepsilon H^{3+\varepsilon} \max _{J \asymp L,|\Delta_{\pm}| <H^{-\frac{1}{2}}}  {\sum}^{*}q\min \left(1, \frac{1}{|\Delta_{\pm} \sqrt{H}|}\right)  \\
			& \ll_{\varepsilon} H^{3+\varepsilon} \max _{J \asymp L,|\Delta_{\pm}| <H^{-\frac{1}{2}}}(\delta+L^{\varepsilon-3/2})\min \left(1, \frac{1}{|\Delta_{\pm} \sqrt{H}|}\right)L^{2}  \\
			& \ll_{\varepsilon} H^{5/2+\varepsilon}L^{3/2}+H^{5/2+\varepsilon}.
		\end{aligned}
		$$
		$ii)\ \delta \asymp|\Delta_{\pm}| L^{-1/2}$, $|\Delta_{\pm}| \geqslant H^{-1/2}$. By Lemma $\ref{9}$ and first derivative estimation,
		$$
		\begin{aligned}
			\int_H^{2 H} {\sum}^{*}qx^2 e\left(\Delta_{\pm} \sqrt{x}\right) d x &\ll_\varepsilon H^{3+\varepsilon} \max _{J \asymp L,|\Delta_{\pm}| \geqslant H^{-\frac{1}{2}}}  {\sum}^{*}q\min \left(1, \frac{1}{|\Delta_{\pm} \sqrt{H}|}\right)  \\
			& \ll_{\varepsilon} 	H^{5/2+\varepsilon} \max_{J \asymp L,|\Delta_{\pm}| \geqslant H^{-\frac{1}{2}}}\left(H^{-\frac{1}{2}} L^{-\frac{1}{2}}+\frac{L^{\varepsilon-\frac{3}{2}}}{|\Delta_{\pm} |}\right)L^{\frac{1}{2}}(NMKRST)^{\frac{1}{4}}\\
			&\ll _{\varepsilon}H^{5/2+\varepsilon}(L^{3/2}+H^{\frac{1}{2}}L^{-1}(NMKRST)^{\frac{1}{4}}).
		\end{aligned}
		$$
		By Lemma $\ref{17}$,we have
		$$
		\begin{aligned}
			A_{\pm}^{\prime\prime}&\ll_{\varepsilon}\prod_{i = 1}^{8}(\delta^{\frac{1}{8}}N_{i}^{\frac{15}{16}}+N_{i}^{\frac{1}{2}})N_{i}^{\varepsilon}\\
			&\ll_{\varepsilon}(\delta^{\frac{1}{8}}N^{\frac{15}{16}}+N^{\frac{1}{2}})(\delta^{\frac{1}{8}}M^{\frac{15}{16}}+M^{\frac{1}{2}})(\delta^{\frac{1}{8}}K^{\frac{15}{16}}+K^{\frac{1}{2}})(\delta^{\frac{1}{8}}L^{\frac{15}{16}}+L^{\frac{1}{2}})\\
			&\ \ \ \ \ \ \ \ (\delta^{\frac{1}{8}}R^{\frac{15}{16}}+R^{\frac{1}{2}})(\delta^{\frac{1}{8}}S^{\frac{15}{16}}+S^{\frac{1}{2}})(\delta^{\frac{1}{8}}T^{\frac{15}{16}}+T^{\frac{1}{2}})(\delta^{\frac{1}{8}}J^{\frac{15}{16}}+J^{\frac{1}{2}})\\
			&\ll_{\varepsilon}\left(\delta^{\frac{1}{4}}(NR)^{\frac{15}{16}}+\delta^{\frac{1}{8}}(R^{\frac{1}{2}}N^{\frac{15}{16}}+R^{\frac{15}{16}}N^{\frac{1}{2}})+(NR)^{\frac{1}{2}}\right)\\
			&\ \ \ \ \  \ \ \left(\delta^{\frac{1}{4}}(MS)^{\frac{15}{16}}+\delta^{\frac{1}{8}}(M^{\frac{1}{2}}S^{\frac{15}{16}}+S^{\frac{1}{2}}M^{\frac{15}{16}})+(MS)^{\frac{1}{2}}\right)\\
			&\ \ \ \ \ \ \ \left(\delta^{\frac{1}{4}}(KT)^{\frac{15}{16}}+\delta^{\frac{1}{8}}(K^{\frac{1}{2}}T^{\frac{15}{16}}+T^{\frac{1}{2}}K^{\frac{15}{16}})+(KT)^{\frac{1}{2}}\right)\\
			&\ \ \ \ \ \ \   \left(\delta^{\frac{1}{4}}(LJ)^{\frac{15}{16}}+\delta^{\frac{1}{8}}(L^{\frac{1}{2}}J^{\frac{15}{16}}+J^{\frac{1}{2}}L^{\frac{15}{16}})+(LJ)^{\frac{1}{2}}\right)\\
			&\ll_{\varepsilon}\delta(NMKLRSTJ)^{\frac{15}{16}}+\delta^{\frac{7}{8}}L^{\frac{15}{8}}(NMKRST)^{\frac{15}{16}}(K^{-\frac{7}{16}}+N^{-\frac{7}{16}}+M^{-\frac{7}{16}}+\\
			&\ \ \ \ \ \ \ \ R^{-\frac{7}{16}}+S^{-\frac{7}{16}}+T^{-\frac{7}{16}})+\delta^{\frac{1}{8}}L^{\frac{23}{16}}(NMKRST)^{\frac{1}{2}}+(NMKLRSTJ)^{\frac{1}{2}}.
		\end{aligned}
		$$
		Taking $\delta \asymp \delta L^{1/2}$,
		$$
		\begin{aligned}
			A_{\pm}^{\prime\prime}
			&\ll_{\varepsilon}\delta(NMKLRSTJ)+\delta^{\frac{7}{8}}L^{\frac{37}{16}}(NMKRST)^{\frac{15}{16}}(K^{-\frac{7}{16}}+N^{-\frac{7}{16}}+M^{-\frac{7}{16}}+\\
			&\ \ \ \ \ \ \ \ R^{-\frac{7}{16}}+S^{-\frac{7}{16}}+T^{-\frac{7}{16}})+\delta^{\frac{1}{8}}L^{3/2}(NMKRST)^{\frac{1}{2}}+(NMKLRSTJ)^{\frac{1}{2}}.
		\end{aligned}
		$$
		Similarly, we have two cases:\\	 $i)\ \delta \asymp|\Delta_{\pm}| L^{-1/2}$, $|\Delta_{\pm}| \geqslant H^{-1/2}$:
		$$
		\begin{aligned}
			\int_H^{2 H} {\sum}^{*}qx^2 e\left(\Delta_{\pm} \sqrt{x}\right) d x &\ll_\varepsilon H^{3+\varepsilon} \max _{J \asymp L,|\Delta_{\pm}| \geqslant H^{-\frac{1}{2}}}  {\sum}^{*}q\min \left(1, \frac{1}{\Delta_{\pm} \sqrt{H}}\right)  \\
			&\ll_{\varepsilon} H^{3+\varepsilon} \max _{J \asymp L,|\Delta_{\pm}| \geqslant H^{-\frac{1}{2}}}\left(A_{\pm} ^{''}/ \Delta_{\pm} \sqrt{H}\right)(N M K L R S T J)^{-3 / 4}\\
			& \ll_{\varepsilon} H^{\frac{5}{2}+\varepsilon}\left[L^{3/2} +(\delta)^{\frac{7}{8}}L^{\frac{5}{16}}(NMKRST)^{\frac{3}{16}}\left(N^{-\frac{7}{16}}+\cdots+T^{-\frac{7}{16}}\right)\right.\\
			&\ \ \ \ \ \ \ \left.+(\delta)^{\frac{1}{8}}L^{-\frac{1}{2}}(NMKRST)^{-\frac{1}{4}}\right]+H^{3+\varepsilon}L^{-\frac{1}{2}}(NMKRST)^{-\frac{1}{4}}.\\
		\end{aligned}
		$$
		$ii)\ \delta \asymp|\Delta_{\pm}| L^{-\frac{1}{2}}$, $|\Delta_{\pm}|<H^{-\frac{1}{2}}$:
		$$
		\begin{aligned}
			\int_H^{2 H} {\sum}^{*}qx^2 e\left(\Delta_{\pm} \sqrt{x}\right) d x &\ll_\varepsilon H^{3+\varepsilon} \max _{J \asymp L,|\Delta_{\pm}| <H^{-\frac{1}{2}}}  {\sum}^{*}q\min \left(1, \frac{1}{|\Delta_{\pm} \sqrt{H}|}\right)  \\
			& \ll_{\varepsilon} H^{3+\varepsilon} \max _{J \asymp L,|\Delta_{\pm}| <H^{-\frac{1}{2}}}\left(A_{\pm} ^{''}\right)(N M K L R S T J)^{-3 / 4}\\
			& \ll_{\varepsilon} H^{3+\varepsilon}\left[ (LH)^{-\frac{7}{16}}L^{\frac{5}{16}}(NMKRST)^{\frac{3}{16}}\left(N^{-\frac{7}{16}}+\cdots+T^{-\frac{7}{16}}\right)\right.\\
			&\ \ \ \ \ \ \ \left.+(LH)^{-\frac{1}{16}}L^{-\frac{1}{2}}(NMKRST)^{-\frac{1}{4}}\right]+H^{3+\varepsilon}L^{-\frac{1}{2}}(NMKRST)^{-\frac{1}{4}}\\
			&\ll_{\varepsilon} H^{5/2+\varepsilon}L^{3/2}+ H^{\frac{41}{16}+\varepsilon}L^{-\frac{1}{8}}(NTS)^{\frac{1}{8}}+ L^{-\frac{9}{16}}H^{\frac{47}{16}+\varepsilon}(NMKRST)^{-\frac{1}{4}}\\
			&\ \ \ \ \ \ \ +H^{3 + \varepsilon}L^{-\frac{1}{2}}(NMKRST)^{-\frac{1}{4}}.\\
		\end{aligned}
		$$
		When $\delta \asymp|\Delta_{-}| L^{-\frac{1}{2}}$ and $|\Delta_{-}|\leqslant H^{-\frac{1}{2}}$, since $|\Delta_{-}| \leqslant H^{-\frac{1}{2}}$ and $|\Delta_{-}| \geqslant L^{-\frac{127}{2}}$, we have $H^{\frac{1}{127}} \ll L\leqslant Y$, thus
		$$
		\begin{aligned}
			\int_H^{2 H} S_6(x) d x &\ll_{\varepsilon}\max _{ H^{1/127}<L\leqslant Y} H^{5/2+\varepsilon}L^{3/2}+H^{3+\varepsilon}L^{-1/2}(NMKRST)^{-\frac{1}{4}}\\
			& \ll_{\varepsilon} H^{3-\frac{1}{254}+\varepsilon} .
		\end{aligned}
		$$
		When $\delta \asymp|\Delta_{-}| L^{-\frac{1}{2}}$ and $|\Delta_{-}|\geqslant H^{-\frac{1}{2}}$,
		\begin{align*}
			\int_{H}^{2H} S_{6}(x)dx &\ll_{\varepsilon}H^{\frac{71}{24}+\varepsilon}+ \frac{H^{\frac{5}{2}+\varepsilon}}{|\Delta_{-}|}\min\left(\frac{(NMKRST)^{\frac{1}{4}}}{L},\frac{|\Delta_{-}|^{\frac{7}{8}}L^{\frac{5}{4}} + 1}{L^{\frac{1}{2}}(NMKRST)^{\frac{1}{4}}}\right)\\
			&\ll_{\varepsilon}H^{\frac{71}{24}+\varepsilon}+ \frac{H^{\frac{5}{2}+\varepsilon}}{|\Delta_{-}|}\left(\frac{(NMKRST)^{\frac{1}{2}}}{L}\right)^{\frac{1}{5}}\left(\frac{|\Delta_{-}|^{\frac{7}{8}}L^{\frac{5}{2}} + 1}{L^{\frac{1}{2}}(NMKRST)^{\frac{1}{4}}}\right)^{\frac{4}{5}}\\
			&\ll_{\varepsilon} H^{\frac{71}{24}+\varepsilon}+H^{\frac{5}{2}+\varepsilon}\left(|\Delta_{-}|^{-\frac{3}{10}}L^{\frac{2}{5}}+L^{-\frac{3}{5}}|\Delta_{-}|^{-1}\right).
		\end{align*}
		When $|\Delta_{-}| \gg L^{-\frac{10}{7}}$:
		$$
		\int_{H}^{2H} S_{6}(x)dx \ll_{\varepsilon} H^{\frac{71}{24}+\varepsilon}+H^{\frac{5}{2}+\varepsilon}L^{-\frac{1}{5}}\ll_{\varepsilon} H^{\frac{71}{24}+\varepsilon}.
		$$
		When $|\Delta_{-}| \ll L^{-\frac{10}{7}}$:
		$$
		\int_{H}^{2H} S_{6}(x)dx \ll_{\varepsilon} H^{\frac{71}{24}+\varepsilon}+H^{\frac{5}{2}+\varepsilon}L^{-\frac{3}{5}}|\Delta_{-}|^{-1}.
		$$
		By Lemma \ref{5}, $|\Delta_{-}| \gg L^{-\frac{127}{2}}$ and $|\Delta_{-}| \gg H^{-\frac{1}{2}}$, so
		\begin{align*}
			\int_{H}^{2H} S_{6}(x)dx &\ll_{\varepsilon} H^{\frac{71}{24}+\varepsilon}+\min\left(H^{\frac{5}{2}+\varepsilon}L^{\frac{389}{20}},H^{3+\varepsilon}L^{-\frac{3}{5}}\right)\\
			&\ll_{\varepsilon} H^{\frac{71}{24}+\varepsilon}+\left(H^{\frac{5}{2}+\varepsilon}L^{\frac{389}{20}}\right)^{\frac{6}{395}}\left(H^{3+\varepsilon}L^{-\frac{2}{5}}\right)^{\frac{389}{395}}\\
			&\ll_{\varepsilon} H^{3 - \frac{1}{254}+\varepsilon}.
		\end{align*}
		Thus $\int_H^{2H} S_6(x) d x \ll_\varepsilon H^{3-\frac{1}{254}+\varepsilon}$, and similarly $\int_H^{2 H} S_5(x) d x \ll_\varepsilon H^{3-\frac{1}{254}+\varepsilon}$.
		
		Following the same method for $\int_H^{2 H} S_5(x) d x$ and $\int_H^{2 H} S_6(x) d x$, we can find the integral means of $\int_H^{2 H} S_2(x) d x$ and $\int_H^{2 H} S_3(x) d x$. By Lemma \ref{15}, Lemma \ref{14}, Lemma \ref{50}, and the above methods, we obtain
		$$
		\begin{aligned}
			& \int_H^{2 H} S_2(x) d x  \ll_{\varepsilon} H^{\frac{3}{2}+\varepsilon} Y^{\frac{3}{2}}+H^{3-\frac{1}{254}+\varepsilon}, \\
			& \int_H^{2 H} S_3(x) d x  \ll_{\varepsilon} H^{\frac{3}{2}+ \varepsilon} Y^{\frac{3}{2}}+H^{3-\frac{1}{254}+\varepsilon}.
		\end{aligned}
		$$
		Finally, by first derivative estimation,
		$$
		\begin{aligned}
			\int_H^{2 H} S_{1}(x) d x & =\frac{1}{128} \int_H^{2 H} \sum q x^2 \cos (4 \pi U\sqrt{x}) \\
			& \ll_{\varepsilon} H^{3+\varepsilon} \sum_{0<n,m, k, l, r, s,t, j \leqslant Y} q \min \left(1,U^{-1}(\sqrt{H})^{-1}\right) \\
			& \ll_{\varepsilon} H^{\frac{5}{2} + \varepsilon} Y^{\frac{3}{2}+\varepsilon}.
		\end{aligned}
		$$
		Therefore,
		$$\int_H^{2 H}\left|{\sum}_Y(x)\right|^8 d x=\left(\frac{35C_7}{128} -\frac{7 C_4}{32}\right) \int_H^{2 H} x^2 d x+O\left(H^{\frac{5}{2}+\varepsilon} Y^{\frac{3}{2}+\varepsilon}\right)+O\left(H^{3-\frac{1}{254}+\varepsilon}\right),$$
		and since $Y\ll H^{11/36}$, the lemma follows.
	\end{proof}
	\section{Proof of Theorem 1.1}
	To calculate the error term of $\int_{H}^{2H}\Delta^{8}(x)dx$, we use Voronoi's truncation formula. Taking $N=H$ in (\ref{103}):
	$$
	\begin{aligned}
		\Delta(x)  =\frac{x^{1/4}}{\pi \sqrt{2}} \sum_{n \leqslant Y} d(n) n^{-3/4} \cos \left(4 \pi \sqrt{n x}-\frac{\pi}{4}\right)+R_{Y , H}(x)+O_{\varepsilon}\left(x^{\varepsilon} \right).\\
	\end{aligned}
	$$
	where $H\leqslant X$. The eighth integral mean of $\Delta(x)$ transforms to:
	$$
	\begin{aligned}
		\int_H^{2 H}\Delta^8(x) d x= &(\pi \sqrt{2})^{-8}\int_H^{2 H}\left|{\sum}_Y(x)\right|^8 d x+O\left(\int_H^{2 H}R_{Y , H}(x)\left|{\sum}_Y(x)\right|^{7} d x  \right)\\&+O\left(\int_{H}^{2H}|R_{Y , H}(x)|^8dx\right).
	\end{aligned}
	$$
	We derive
	\begin{align*}
		\int_{H}^{2H}R_{Y,H}^2(x)dx &\ll_{\varepsilon} H^{1 + \epsilon}+\int_{H}^{2H}\left|x^{1/4}\sum_{Y < n \leq H}d(n)n^{-3/4}e(2\sqrt{nx})\right|^2dx\\
		&\ll_{\varepsilon} H^{1 + \epsilon}+H^{3/2}\sum_{Y < n \leq H}d^2(n)n^{-3/2}+H\sum_{Y < m < n \leq H}\frac{d(m)d(n)}{(mn)^{3/4}(\sqrt{n}-\sqrt{m})}\\
		&\ll_{\varepsilon} \frac{H^{3/2}\log^3H}{Y^{1/2}},
	\end{align*}
	using the estimates
	$$\sum_{n \leq v}d^2(n)\ll v\log^3v$$
	and
	$$\sum_{Y < m < n \leq H}\frac{d(n)d(m)}{(nm)^{3/4}(\sqrt{n}-\sqrt{m})}\ll H^{\varepsilon}.$$
	Take $A_0 = 267/27$ and $H^{\varepsilon}\leq Y \leq H^{11/36}$. For any $2 < A < A_0$, by Lemma \ref{19} and Hölder's inequality:
	\begin{align*}
		\int_{H}^{2H}|R_{Y,H}(x)|^Adx&=\int_{H}^{2H}|R_{Y,H}(x)|^{\frac{2(A_0 - A)}{A_0 - 2}+\frac{A_0(A - 2)}{A_0 - 2}}dx\\
		&\ll_{\varepsilon} \left(\int_{H}^{2H}|R_{Y , H}(x)|^2dx\right)^{\frac{A_0 - A}{A_0 - 2}}\left(\int_{H}^{2H}|R_{Y , H}(x)|^{A_0}dx\right)^{\frac{A - 2}{A_0 - 2}}\\
		&\ll_{\varepsilon} H^{1 + A/4+\varepsilon}Y^{-\frac{A_0 - A}{2(A_0 - 2)}}.
	\end{align*}
	In particular, when $A = 8$ and $Y = H^{11/36}$:
	\begin{align*}
		\int_{H}^{2H}|R_{Y , H}(x)|^8dx&\ll_{\varepsilon} H^{3+\varepsilon}Y^{-\frac{267/27- 8}{2(267/27- 2)}}\\
		&\ll_{\varepsilon} H^{3 - 187/5112+\varepsilon}\\
		&\ll_{\varepsilon} H^{3 - 1/254+\varepsilon}.
	\end{align*}
	By Hölder's inequality:
	\begin{align*}
		\int_H^{2 H}R_{Y , H}(x)\left|{\sum}_Y(x)\right|^{7} d x &\ll \left(\int_{H}^{2H}|{\sum}_Y(x)|^{8}dx\right)^{\frac{7}{8}}\left(\int_{H}^{2H}|R_{Y , H}(x)|^{8}dx\right)^{\frac{1}{8}}\\
		&\ll_{\varepsilon} (H^{3+\varepsilon})^{\frac{7}{8}}(H^{3 - 187/5112+\varepsilon})^{\frac{1}{8}}\\
		&\ll_{\varepsilon} H^{3+\varepsilon - 187/40896}\\
		&\ll_{\varepsilon} H^{3 - 1/254+\varepsilon}.
	\end{align*}
	Combining with the estimate of $\int_H^{2H}\left|{\sum}_Y(x)\right|^8 d x$ from Lemma \ref{72}, we get:
	$$\int_H^{2 H}\Delta^8(x) d x=(\pi \sqrt{2})^{-8}\left(\frac{35C_7}{128} -\frac{7 C_4}{32}\right) \int_H^{2 H} x^2 d x+O\left(H^{3-\frac{1}{254}+\varepsilon}\right).$$
	Hence,
	$$
	\begin{aligned}
		\int_2^{X}\Delta^8(x) d x&=(\pi \sqrt{2})^{-8}\left(\frac{35C_7}{128} -\frac{7 C_4}{32}\right)\sum_{i=1}^{N} \int_{2^{i}}^{2^{i+1} } x^2 d x+O\left((2^{N+2})^{3-\frac{1}{254}+\varepsilon}\right)\\
		&=(\pi \sqrt{2})^{-8}\left(\frac{35C_7}{128} -\frac{7 C_4}{32}\right) \int_{2}^{X} x^2 d x+O\left(X^{3-\frac{1}{254}+\varepsilon}\right),
	\end{aligned}
	$$
	where $2^{N} \leqslant X \leqslant2^{N+1}$, thus proving Theorem \ref{80}.	
	
	\section{Proof of Theorem 1.2}
	It suffices to calculate $\int_X^{X+H} \Delta^8(x) d x=(\pi \sqrt{2})^{-8} \int_X^{X+H} |{\sum}_{Y}(x)|^8 d x+O\left(\sum\limits_{i=1}^8 I_i\right)$, where
	$$
	I_i:=\int_X^{X+H}\left|R_{Y , X}(x)\right|^i\left|{\sum}_Y(x)\right|^{8-i} d x.
	$$
	Theorem \ref{81} can be obtained from the following lemmas.
	\begin{lemma}\label{42}
		Let $$S=\sum\limits_{N\leq n\leq 2N} e(f(n))$$ $q\geq 0$, and $f$ be $(q+2)$ times continuously differentiable on $I\subset (N,2N]$. If there exists a constant $F$ such that
		$$|f^{(r)}|\approx FN^{-r}$$
		for $r=1,\cdots,q+2$ and $Q=2^q$, then
		$$S\ll F^{1/(4Q-2)}N^{1-(q+2)/(4Q-2)}+FN^{-1}.$$
	\end{lemma}
	\begin{proof}
		See Lemma 2.9 in Kolesnik and Graham \cite{1}.
	\end{proof}
	\begin{lemma}\label{20}
		For any sufficiently small $0<k<\frac{1}{2}$ , $X^{2k} \leq Y \ll X$, $H\leqslant X$, we have
		$$
		\begin{aligned}
			& \int_X^{X+H} R_{Y,X}^8(x) d x \ll_k X^{355/160-k}+H X^{2-2k/3},
		\end{aligned}
		$$
		where $R_{Y, X}(x)=x^{1/4} \sum\limits_{Y<n \leqslant X} d(n) n^{-3/4} \cos \left(4 \pi \sqrt{n x}-\frac{\pi}{4}\right)$.
	\end{lemma}
	\begin{proof}
		Take $Y\leqslant L\leqslant X$, then
		$$
		\begin{aligned}
			\int_X^{X+H} R_{L, X}^2(x) dx 
			& \ll_{\varepsilon}X^{1/2+\varepsilon}\int_X^{X+H}\left| \sum_{L< n_1,n_2\leqslant K}d(n_1)d(n_2) (n_1n_2)^{-3 / 4} e^{4 \pi i \sqrt{x}(\sqrt{n_1}-\sqrt{n_2})}\right|dx\\
			&\ll_\varepsilon X^{1+\varepsilon}L^{\varepsilon-3/2}\left| \sum_{L< n_1\neq n_2\leqslant K}\frac{1}{(\sqrt{n_1}-\sqrt{n_2})}\right|+HX^{1/2+\varepsilon}L^{\varepsilon-1/2}\\
			&\ll_\varepsilon L^{\varepsilon-1} X^{1+\varepsilon} +HX^{1/2+\varepsilon}L^{\varepsilon-1/2}.\\
		\end{aligned}
		$$
		By Lemma \ref{19} and Hölder's inequality, take $A_0 = 267/27$. For any $2 < A < A_0$, we have
		\begin{align*}
			\int_{X}^{X+H}|R_{L,X}(x)|^Adx&=\int_{X}^{X+H}|R_{L,X}(x)|^{\frac{2(A_0 - A)}{A_0 - 2}+\frac{A_0(A - 2)}{A_0 - 2}}dx\\
			&\ll_{\varepsilon} \left(\int_{X}^{X+H}|R_{L , X}(x)|^2dx\right)^{\frac{A_0 - A}{A_0 - 2}}\left(\int_{X}^{X+H}|R_{L , X}(x)|^{A_0}dx\right)^{\frac{A - 2}{A_0 - 2}}\\
			&\ll_{\varepsilon} HX^{A/4+\varepsilon}L^{-\frac{A_0 - A}{2(A_0 - 2)}}.
		\end{align*}
		Take $A=8$,
		\begin{align*}
			\int_{X}^{X+H}|R_{L,X}(x)|^8dx\ll_{\varepsilon} HX^{2+\varepsilon}L^{-\frac{51}{423}}.
		\end{align*}
		On one hand, we simply compute
		$$ R_{L, K}(x) \ll_\varepsilon X^{1/4+\varepsilon}\left| \sum_{L< n\leqslant K}d(n) n^{-3 / 4} e^{4 \pi i \sqrt{nx}}\right| $$
		and
		$$ R_{L, K}(x) \ll_\varepsilon X^{1/2+\varepsilon}L^{-1/2}. $$
		In Lemma \ref{42}, take $q=3$, $F=\sqrt{xL}$, $N=L$:
		$$
		\begin{aligned}
			R_{L, K}(x) &\ll_\varepsilon \min\left(X^{1/4+\varepsilon}L^{-3/4}  \left| \sum_{L< n\leqslant K}  e^{4 \pi i \sqrt{nx}}\right|,X^{1/2+\varepsilon}L^{-1/2}\right) \\
			& \ll_\varepsilon \left(X^{1/4+\varepsilon}L^{1/10}x^{1/60} \right)^{\frac{3}{4}}\cdot\left(X^{1/2+\varepsilon}L^{-1/2}\right)^{\frac{1}{4}} \\
			& \ll_\varepsilon X^{13/40+\varepsilon} L^{-1/20}.
		\end{aligned}
		$$
		On the other hand,
		$$
		\begin{aligned}
			\int_X^{X+H} R_{L, K}^8(x) dx &\ll_\varepsilon\left(X^{13/40+\varepsilon} L^{-1/20}\right)^6 \int_X^{X+H} R_{L,K}^2(x) d x \\
			&\ll_\varepsilon X^{59/20+\varepsilon} L^{\varepsilon-13/10}+HX^{49/20+\varepsilon}L^{\varepsilon-4/5}.
		\end{aligned}
		$$
		If $Y \leqslant L \leqslant X^{9/16+k}$:
		$$
		\begin{aligned}
			\int_X^{X+H} R_{L, K}^8(x) d x 
			&\ll_\varepsilon H X^{2+\varepsilon-4k/5 }.\\
		\end{aligned}
		$$
		If $L>X^{9/16+k}$:
		$$
		\begin{aligned}
			\int_X^{X+H} R_{L, K}^8(x) d x \ll_\varepsilon X^{355/160 +\varepsilon-13k/10}  +H X^{2+\varepsilon-2k/3 }.
		\end{aligned}
		$$
		With $\varepsilon, k$ arbitrarily small, the lemma follows by choosing appropriate relations between them.
	\end{proof}
	From the above discussion,
	$$\int_X^{X+H} |{\sum}_Y(x)|^8 d x= C\left((X+H)^3-X^3\right)+O\left(H^{\frac{3}{2}+\varepsilon} X Y^{\frac{1}{2}+\varepsilon}\right),$$
	where $$C=\frac{1}{3}(\pi \sqrt{2})^{-8}\left(\frac{35C_7}{128} -\frac{7 C_4}{32}\right).$$
	In particular, setting $Y=X^k$ for sufficiently small $k>0$, we have
	$$
	\int_X^{X+H} {\sum}_Y^8(x) d x \ll H X^2 \quad\left(X^{7/32+\delta} \leqslant H \ll X\right).
	$$
	By Lemma \ref{20} and Hölder's inequality, for $i=1,2,\dots,8$:
	$$
	\begin{aligned}
		I_i	& =\left(\int_X^{X+H} R_{Y, X}^8 d x\right)^{i / 8}\left(\int_X^{X+H} {\sum}_Y^8(x) d x\right)^{(8-i) / 8} \\
		& \ll\left(X^{355/160-k}+H X^{2-2k/3}\right)^{i / 8}\left(HX^2\right)^{(8-i) / 8} \\
		& \ll H X^{2-ki/12}+H^{1-i/8} X^{2+35i/1280-ki/8}.
	\end{aligned}
	$$
	If $H \geqslant X^{7/32+\delta}$, for a fixed $\delta>0$, take $0<k<\delta$:
	$$
	I_i \ll H X^{2-k/12} \quad(i=1,2,\dots,8),
	$$
	thus
	$$\begin{aligned}
		\int_X^{X+H} \Delta^8(x) d x=C\left((X+H)^3-X^3\right)\left(1+O\left(X^{-k}\right)\right)&\\
		&\left(X^{7/32+\delta} \leqslant H\leqslant X\right).
	\end{aligned}
	$$
	
	\noindent
	\textbf{Ackonwledgement}.  The author would like to express sincere gratitude to Professor Liu Zhiguo from East China Normal University for his valuable input and consistent encouragement. The author   is also grateful to Dr. Liao Junjie for his invaluable guidance, insightful discussions, and continuous support throughout the research. 
	Additionally, heartfelt thanks are extended to Professors Cai Yingchun and Zhou Haigang from Tongji University for their generous academic support. This study was not supported by any funding.






\end{document}